\documentclass[pdflatex,sn-mathphys-num]{sn-jnl}

\usepackage{graphicx}%
\usepackage{multirow}%
\usepackage{amsmath,amssymb,amsfonts}%
\usepackage{amsthm}%
\usepackage{mathrsfs}%
\usepackage[title]{appendix}%
\usepackage{xcolor}%
\usepackage{textcomp}%
\usepackage{manyfoot}%
\usepackage{booktabs}%
\usepackage{algorithm}%
\usepackage{algorithmicx}%
\usepackage{algpseudocode}%
\usepackage{listings}%

\newtheorem{theorem}{Theorem}
\newtheorem{proposition}{Proposition}
\newtheorem{definition}{Definition}%
\newtheorem{assumption}{Assumption}%

\def\C{{\mathbb C}}
\def\R{{\mathbb R}}
\def\X{{\mathbb X}}
\def\Y{{\mathbb Y}}
\def\H{{\mathcal H}}

\def\Re{{\mathsf{Re}}}
\def\HH{{\bold{ H}}}

\newcommand{\NN}{\mathbb{N}}
\newcommand{\bx}{\mathbf{x}}

\newcommand{\dx}{\Delta {x}}

\newcommand{\by}{\mathbf{y}}

\newcommand{\da}{\Delta a}

\newcommand{\FF}{\mathcal{F}}

\graphicspath{{./Figures/}}

\begin{document}


\title{The bilinear Hessian for large scale optimization}


\author*[1]{\fnm{Marcus} \sur{Carlsson}}\email{marcus.carlsson@math.lu.se}
\equalcont{These authors contributed equally to this work.}

\author[2]{\fnm{Viktor} \sur{Nikitin}}\email{vnikitin@anl.gov}
\equalcont{These authors contributed equally to this work.}

\author[1]{\fnm{Erik} \sur{Troedsson}}\email{erik.troedsson@math.lu.se}
\equalcont{These authors contributed equally to this work.}

\author[3]{\fnm{Herwig} \sur{Wendt}}\email{herwig.wendt@irit.fr}
\equalcont{These authors contributed equally to this work.}

\affil[1]{\orgdiv{Centre for Mathematical Sciences}, \orgname{Lund University}, \orgaddress{\street{Sölvegatan 18}, \city{Lund}, \postcode{223 62}, \country{Sweden}}}

\affil[2]{\orgdiv{Advanced Photon Source}, \orgname{Argonne National Laboratory}, \orgaddress{\street{9700 S Cass Ave}, \city{Lemont}, \postcode{60439}, \state{Illinois}, \country{United States}}}

\affil[3]{\orgdiv{IRIT Laboratory}, \orgname{Universit\'{e} de Toulouse}, \orgaddress{\street{2 rue Camichel}, \city{Toulouse}, \postcode{31071}, \country{France}}}


\abstract{Second order information is useful in many ways  in smooth optimization problems, including for the design of step size rules and descent directions, or the analysis of the local properties of the objective functional. However, the computation and storage of the Hessian matrix using second order partial derivatives is prohibitive in many contexts, and in particular in large scale problems. 
In this work, we propose a new framework for computing and presenting second order information in analytic form. The key novel insight is that the Hessian for a problem can be worked with efficiently by computing its \emph{bilinear form} or \emph{operator form} using Taylor expansions, instead of introducing a basis and then computing the Hessian \textit{matrix}. Our new framework is suited for high-dimensional problems stemming e.g.~from imaging applications, where computation of the Hessian matrix is unfeasible. We also show how this can be used to implement Newton's step rule, Daniel's Conjugate Gradient rule, or Quasi-Newton schemes, without explicit knowledge of the Hessian matrix, and illustrate our findings with a simple numerical experiment.}




\maketitle

\section{Introduction}\label{sec1}
Second order information is seldom used  in smooth optimization for large scale applications because the Hessian matrix, traditionally computed by use of second order derivatives, becomes too large to handle. 
For example, in 3D tomography it is not unusual to want to recover ``data-cubes'' of the size $1000^3$, which leads to a Hessian matrix of size $10^9\times 10^9$. But even for more modest 2D imaging applications where the amount of unknowns can be as small as $500^2$, computing and storing the Hessian (involving roughly $500^4/2\approx 3\cdot 10^{10}$ second order partial derivatives) becomes prohibitively slow. As a consequence, first order methods relying on gradients, i.e.~first order derivatives of the objective functional only, or even alternating projection type schemes, are still predominant, see e.g.~\cite{elser2003phase,fienup1982phase} and the references therein.
The usefulness of second order information would have in smooth optimization is, however, undisputed, so that this situation presents a significant deadlock.

To give concrete examples for this, let $\X$ be some (possibly complex valued) linear space (typically matrices or tensors) and let $f:\X\rightarrow \R$ be the function to be minimized via some descent algorithm. Here we allow $f$ to be non-convex in which case the algorithm is only expected to converge to a local minimizer. 
Even if $\X$ is a linear space with complex entries, we can introduce a basis over $\R$ so that each element $x$ can be identified with a vector of basis coefficients $\bold{x}$ in $\R^n$, and we can treat $f$ as a functional on $\R^n$. 
Let $x_k\in\X$ be a current iterate and let $s_k$ be the new search direction. 
Given $s_k$, a common problem is to pick an appropriate step length $\alpha_k$, ideally so that $f(x_k+\alpha_k s_k)$ is near the minimum value along the graph $\alpha\mapsto f(x_k+\alpha s_k)$. One solution to this is to fix the step length, leading to suboptimal convergence, another is to do a line-search using e.g.~bisection methods, which is time consuming because it requires numerous objective functional evaluations, yet another popular method is to use Goldstein-Armijo rules. These rules are not constructed to achieve near-optimal convergence rates, but only ensure that the step length choice is good enough for linear convergence (see e.g.~Section 1.2.3 of \cite{nesterov2013introductory}). Alternatively,  \textit{if} we knew the gradient $\nabla f|_{\bx_k}$ and the Hessian matrix $\HH|_{\bx_k}$ of $f$ at $\bx_k$, we could use the quadratic approximation
\begin{equation}\label{secprx}
  f(\bx_k+\alpha \bold s_k)\approx f(\bx_k)+\alpha \langle \nabla f|_{\bx_k},  \bold{s}_k\rangle+\frac{\alpha^2}{2}\langle \HH|_{x_k} \bold s_k,\bold s_k \rangle 
\end{equation}
to find a close to optimal value of $\alpha_k$ by setting 
$
  \alpha_k= -\frac{\langle \nabla f|_{\bx_k},\bold s_k\rangle}{\langle \HH|_{\bx_k}\bold s_k,\bold s_k \rangle}.
$
This step-size rule is known as Newton's rule, but it is rarely used in practice and, for example it is not even mentioned in the influential classics \cite{boyd2004convex,nesterov2013introductory}.
A second example is provided by the conjugate gradient (CG) method, where one would like to compute the real number $\langle \HH|_{\bx_k}\bold s_k,\bold t_k \rangle$ for two different vectors $\bold s_k$ and $\bold t_k$, which would lead to a principled strategy for computing conjugate directions that is known as Daniel's rule \cite{daniel1967conjugate} (cf. \eqref{eqbetaD} below for details). The use of the Hessian matrix in CG methods is, however, in general discarded upfront with the argument that the cost for computing and storing it would annihilate any computational gain. Instead, the practitioner is left with picking one among many heuristic rules for constructing the conjugate directions at each step, see; e.g., the overview article \cite{hager2006survey} and the recent book \cite{andrei2020nonlinear}.
A third example is given by second order (Quasi-)Newton algorithms, where one is faced with (approximately) solving an equation of the form 
\begin{equation}\label{QNeq}
  -\nabla f|_{\bx_k}=\HH|_{\bx_k} \bold s_k
\end{equation}
for the unknown Newton descent direction $\bold s_k$.

One way out of this dilemma is to make use of the realization that none of the above expressions requires knowledge of the Hessian \emph{matrix}, but instead only computation of one real number $\langle \HH|_{\bx_k} \bold s_k,\bold s_k \rangle$ or $\langle \HH|_{\bx_k} \bold s_k,\bold t_k \rangle$, or one vector $\HH|_{\bx_k} \bold s_k$. 
This was observed in, e.g., \cite{pearlmutter1994fast} where it was used to provide a Hessian-matrix-free algorithm to compute Hessian-vector products as in \eqref{QNeq} in the context of neural networks. The key tool used therein and in subsequent works is Automatic Differentiation (AD), a technique popular in machine learning for computing gradients without needing to derive explicit formulas. {However, AD may significantly slow down computations compared to evaluating direct formulas, a point that we will further elaborate on in Section \ref{secAD}.} 

In this work, we also build on this realization but propose to follow a different, even more computationally efficient, path to extract second order information without having to compute the Hessian matrix. Our approach does not require the use of AD-specific toolboxes (such as PyTorch) but instead makes it possible to easily identify the \emph{analytic expressions} for the above quantities, so that the corresponding routines can be implemented directly.
Specifically, in this work, we present a method for computing the \emph{bilinear Hessian}, which for each fixed $x\in\X$ is a unique symmetric bilinear form $\H|_{x}$ on $\X$ such that
\begin{equation}\label{hes}
  {\H|_{x}(s,t)}=\langle \HH|_{\bx}\bold s,\bold t \rangle
\end{equation}
holds for all $s,t\in\X$ (where $\bold s,~\bold t$ denotes the corresponding coefficients in some given basis), and can be efficiently computed for large scale applications as those mentioned above. 
Moreover, we use this representation to extract a set of efficiently implementable rules that combined form a self-adjoint linear operator acting on $\X$, which we call the \emph{Hessian operator} $H|_{x}$, with the property that $H|_{x}(s)$ is the element in $\X$ corresponding to the coefficients $\HH|_{\bx}\bold s$. A key point of our approach is to \emph{avoid introducing a basis} (or column-stacking type operations), thereby identifying the linear space $\X$ with $\R^n$, and to instead work with Taylor expansions directly on $\X$: 
Indeed, upon identification with $\R^n$, one would naturally end up with the use of the chain rule, and Wirtinger type derivatives in the complex case. This would not only lead to a significant notational challenge for the second order derivatives even for simple problems, but also to complicated and error prone expressions 
with summing operations that quickly become prohibitively slow. With our method, we easily arrive at simple expressions that are naturally expressed in the operators that are inherent in the construction of the functional $f$. 
We illustrate this with a simple example throughout the text. For more advanced examples we refer to \cite{troedsson2024joint} (joint eigendecomposition of matrices) and \cite{carlsson2025efficient} (ptychography).

The necessary preliminaries for our method are gathered in Section \ref{sec:taylor}, including details on how to deal with the situation when the underlying space $\X$ is complex. Our proposed approach is described and studied in general in Section \ref{secghb}, which culminates with a schematic description of the key steps needed for computing the sought first and second order information without ever ``leaving'' the space $\X$. 
In Section \ref{sec_lg}, we extend this to the definition of efficient Newton step size rules, Conjugate Gradient methods, and Quasi-Newton schemes. Throughout, we make use of a simple pedagogic example, for which we report numerical results and comparisons in Section \ref{secnum}. 
In Section \ref{secloc}, we provide theorems proving convergence to local minima of Gradient Descent (GD) with the concrete rules for step-size suggested in this paper. Section \ref{sec13} concludes on this work and points to future research.

\section{Taylor's formula in inner product spaces}
\label{sec:taylor}

In most applications of inverse problems to imaging and signal processing, the sought object is naturally represented as a multiarray $(x_\gamma)_{\gamma\in \Gamma}$, where $\Gamma$ can be any subset of $\mathbb{Z}^d$ for some $d\in \mathbb{N}$ and each $x_j$ denotes either a real or a complex number. Denote the space of all such multi-arrays by $\X$ and suppose we are interested in finding a local minimum of a functional $f:\X\rightarrow \R$. Upon sorting the multiarray into a vector, one can identify $\X$ with either $\R^n$ or $\C^n$ where $n$ equals the cardinality of $\Gamma$, and then, at least in the real case, one can think of $f$ as a functional acting on $\R^n$. This allows us to employ multivariable calculus, compute gradients and, where attainable Hessian matrices, which can be used in optimization routines such as Gradient Descent, Conjugate Gradient descent or second order Newton-type methods.
However, in modern applications, $n$ is often a huge number and the above scheme becomes untractable or even impossible to execute. In this section and the next we discuss how to avoid such issues using tools from basic linear algebra, and also how to efficiently deal with the complex case. In particular, we show how to avoid introducing a basis and do all operations in $\X$. The material in this chapter is not new, for example it can be found in the classic \cite{cartan1971differential} which treats the more general situation of a Banach space.

\subsection{Inner product spaces}

First note that $\X$, as introduced above, naturally becomes an \textit{inner product space} endowed with the inner product  $$\langle x,y\rangle=\sum_{\gamma\in \Gamma} x_\gamma\overline {y_\gamma},$$ where the bar denotes the complex conjugate. An inner product space is simply any linear space (over $\R$ or $\C$) endowed with an inner product (also commonly referred to as a \textit{scalar product}). 
Moreover, in the complex case, we can always consider such a linear space (over $\C$) as a linear space over $\R$ with the corresponding real inner product 
\begin{equation}\label{realSP}
\langle x, y\rangle_{\R}:= \mathsf{Re}\langle x,y\rangle.
\end{equation} For example, $\C^n$ with the usual inner product, is also a linear space over $\R$ (of dimension $2n$) and the above formula gives $$\langle u+iv,x+iy\rangle_{\R}=\sum_j (u_j x_j+v_j y_j)$$ which is precisely the standard inner product on $\R^{2n}$. 
\begin{assumption}\label{def1}
  In the entire paper, $\X$ will denote a fixed finite dimensional inner product space, either over $\R$ or $\C$. In the latter case, we will still consider it a space over $\R$ with the scalar product \eqref{realSP}.
\end{assumption} 
We remark that the concept of a linear space is more general than the multi-array type spaced alluded to above, for example $\X$ can also be the space of real or complex Hermitian matrices.

By basic linear algebra, every finite dimensional inner product space 
 with dimension $n$ has
a basis with $n$ linearly independent vectors $\{e_1,\ldots,e_n\}$, which can be taken orthonormal, i.e.~such that $\|e_j\|=1$ for each $j$ and $\langle e_j,e_k\rangle_{\R}=0$ whenever $j\neq k$. Here we use the subindex ${\R}$ in the scalar product to underscore that even complex spaces are considered as real inner product spaces. Thus, for example $\C^2$ can be seen as a real inner product space {with dimension $n=4$ and} orthonormal basis e.g.~$(1,0),~(i,0),~(0,1),~(0,i)$.

\subsection{Smooth functions between inner product spaces}

Let $\X$ be as in Assumption \ref{def1} and let $\{e_1,\ldots,e_n\}$ be an orthonormal basis. We can then define the isometric isomorphism $\iota_\X:\R^n\rightarrow \X$ by \begin{equation}\label{iota}
                                        \iota_\X\big(\bx\big)=\sum_{j=1}^n \bx_je_j,
\end{equation} 
where $\bx=(\bx_j)_{j=1}^n$, and we can define smoothness of functions $f:\X\rightarrow \R$ in terms of their pullbacks $\iota_\X^{-1}(f)$. Clearly, $\iota_\X$ depends on the choice of orthonormal basis, but given two such maps $\iota_{\X,1}$ and $\iota_{\X,2}$ we have that $\iota_{\X,1}^{-1}\iota_{\X,2}$ is an isometry on $\R^n$ and hence given by an orthogonal matrix $U$. Since $\iota_{\X,2}=\iota_{\X,1}\iota_{\X,1}^{-1}\iota_{\X,2}=\iota_{\X,1} U$, we see that the above definition is basis independent.

In many practical applications, the objective functional $f: \X\rightarrow \R$ is a composition of several non-linear functions, where only the last one has $\R$ as codomain. It is therefore convenient to develop a more general theory for functions $\FF:\X\rightarrow \Y$ between different inner product spaces $\X$ and $\Y$. Following standard approaches in e.g.~differential geometry \cite{lang1995differential}, we shall use the following definition.
\begin{definition}\label{def2}
Let $\X$ and $\Y$ be inner product spaces as in Assumption \ref{def1}, with dimensions $n$ and $m$ respectively, and let $\FF:\X\rightarrow \Y$ be a function. We then say that $f$ is $K$-smooth, or shortly that $\FF\in C^K(\X,\Y)$, if $\iota_\Y^{-1} \FF \iota_\X$ is $K$ times continuously differentiable as a function from $\R^n$ to $\R^m$ in the classical sense of multivariable calculus.  
\end{definition} 
As before we see that the definition is basis independent, but again, the point of this section is to skip the bases and work directly in $\X$ and $\Y$. The below definition introduces linear operators, which generalizes the idea of a matrix to a basis free definition. Figure \ref{fig:illustration1} explains the idea.
\begin{wrapfigure}{r}{0.3\textwidth}
\centering
\includegraphics[width=1\linewidth]{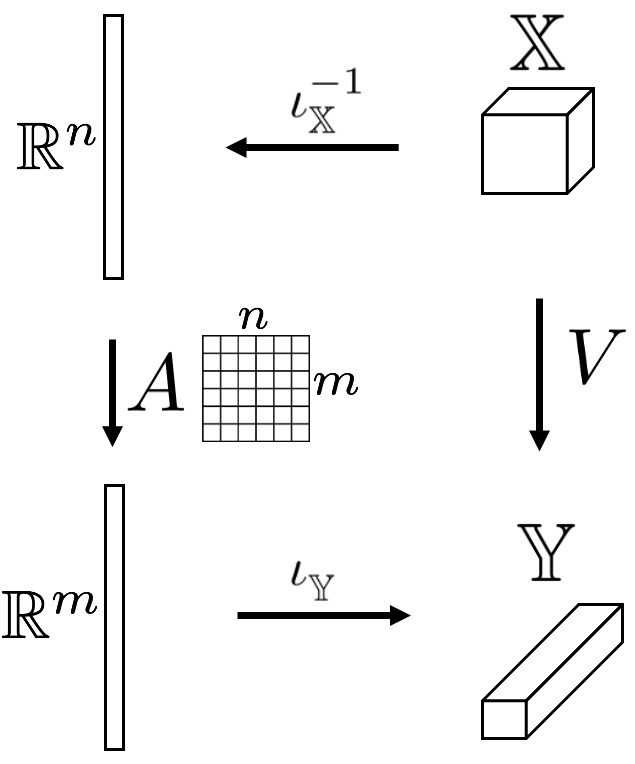}
\caption{\label{fig:illustration1} Given a fixed basis in $\X$ and $\Y$, all linear operators can be represented as matrices, and vice versa. For practical implementation, it is desirable to avoid matrix representations of linear operators (e.g., using FFT to evaluate Fourier transforms).}
\end{wrapfigure}
\begin{definition}\label{def3}
Let $\X$ and $\Y$ be inner product spaces as in Assumption \ref{def1}. A function $V:\X\rightarrow \Y$ is called linear if $V(ax+by)=aV(x)+bV(y)$ holds for all $x,y\in\X$ and $a,b\in\R$. Similarly, a function $W:\X\times \X \rightarrow\Y$ is called bilinear if it is linear in its first and second argument individually, keeping the other one fixed. Finally, $W$ is said to be symmetric if $W(x,y)=W(y,x)$ for all $x,y\in\X$.
\end{definition}
If $\Y=\R$ we usually call $V$ a linear \textit{functional} and denote it by $f$. Otherwise we usually refer to $V$ and $W$ as linear and bilinear \textit{operators}, respectively. We remark that $V$ is linear if and only if $\iota_\Y^{-1} V\iota_\X(\bx)$ can be represented as a matrix-vector multiplication $\bold A\bx$, as illustrated in Figure \ref{fig:illustration1}. 
Also note that if $\X$ and $\Y$ are linear spaces over $\C$, and $V$ is linear over $\C$, then it is automatically linear over $\R$ as well. The same goes for complex bilinear and sesqui-linear forms and operators.

We are now ready to state the main result of this section.
Since we do not strive for maximal generality but rather 
to expose in simple terms our approach
for avoiding the use of bases and partial derivatives, we satisfy with stating the theorem for functions that are $C^2$. 

\begin{theorem}\label{t1}
  Let $\FF:\X\rightarrow \Y$ be a $C^2$-function, where $\X$ and $\Y$ are spaces as in Assumption \ref{def1}. Then, for each fixed $x\in\X$, there exists a unique linear operator $d\FF|_x:\X\rightarrow \Y$ and a unique bilinear symmetric operator $d^2\FF|_x:\X\oplus\X\rightarrow \Y$ such that \begin{equation}\label{dfgs}
  \FF(x+\Delta x)=\FF(x)+d\FF|_x(\Delta x)+\frac{1}{2}d^2\FF|_x(\Delta x,\Delta x)+o(\|\Delta x\|^2),\quad \Delta x\in\X
                                                                                       \end{equation} 
where $o$ denotes ``little ordo''.                                                                                        
\end{theorem}

\begin{proof}
Let $\{e_k\}^m_{k=1}$ be a basis in $\Y$ and consider the function $$f_k(\bx)=\langle \FF(\iota_{\X}(\bx)),e_k\rangle_\R.$$ By the multivariable version of Taylors theorem, at each fixed $\bx\in\R^n$, there exists a gradient $\nabla f_k|_\bx$ and a symmetric Hessian matrix $\HH_{k}|_{\bx}$ such that 
$$  f_k(\bx+\Delta \bx)=f_k(\bx)+\langle \nabla f_k|_\bx,\Delta \bx\rangle_\R+\frac{1}{2} \langle \HH_{k}|_{\bx}\Delta \bx,\Delta \bx\rangle_{\R}+o(\|\Delta \bx\|^2),\quad \Delta \bx\in\R^n.$$
The sought real linear functional, for $x=\iota_{\X}(\bx)$, is now given by $$d\FF|_x(y)=\sum_{k=1}^m \langle \nabla f_k|_\bx,\iota_{\X}^{-1}(y) \rangle_\R e_k$$
and similarly, the real bilinear Hessian is given by 
$$d^2\FF|_x(y,z)=\sum_{k=1}^m \langle \HH_{k}|_{\bx}\iota_{\X}^{-1}(y),\iota_{\X}^{-1}(z)\rangle_{\R} e_k.$$
The uniqueness is easily derived from the uniqueness of the gradient and Hessian in the multivariable Taylor's formula.
\end{proof}

It is the uniqueness statement of the above theorem which will enable us to work with first and second order information without performing any explicit differentiation in some basis. The reason for this is that most objective functionals we wish to optimize are composed of a number of basic $C^\infty$-functions (and linear operators) which each have a well-known Taylor expansion that can be used to derive $V=d\FF|_x$ and $W=d^2\FF|_x$ directly, see e.g.~\cite{carlsson2025efficient,troedsson2024joint} as well as Section \ref{sec:poisson1} for examples. However, while it is often easy to find explicit formulas for $d\FF|_x$ and $d^2\FF|_x$, these are by themselves of limited use. What is really needed is versions of the gradient and Hessian that work directly on $\X$ and $\Y$, \emph{thus avoiding any bases}. In the next section we define these objects and demonstrate how to use $d\FF|_x$ and $d^2\FF|_x$ to find them. To make things a bit more concrete we first consider an example.

\subsection{An example: Poisson noise}
\label{sec:poisson1}

Given $N\in\mathbb{N}$, let $\R^{N\times N}$ denote $\R^N\otimes \R^N$, i.e.~$N\times N$ matrices with real entries, and consider a linear system $T:\R^{N\times N}\rightarrow \R^{N\times N}$ which maps the non-negative orthant $\R^{N\times N}_+$ into itself (representing, e.g., convolution with a point spread function). Let us furthermore suppose that the target space $\R^{N\times N}$ represents a detector, and that the events impinging on the system (detector and operator $T$) in a given time interval can be modeled as a Poisson process with rate parameters $x\in \R^{N\times N}$. The measurements $c\in\NN^{N\times N}$ on the detector (e.g., photon counts) are then modeled as a Poisson process with rate parameter $y_\gamma=\big(T(x)\big)_\gamma$, i.e., a probability mass function $P(c_\gamma)\propto y_\gamma^{c_\gamma}e^{-y_\gamma}$, where $\gamma=(j_1,j_2)\in \Gamma$ with $\Gamma=\{1,\ldots,N\}^2$. 

Suppose that we are given $c$ and we want to estimate $x\in \R^{N\times N}$. The likelihood and negative log-likelihood for $x$ is thus 
\begin{align*}
LH(x)&\propto \prod_{\gamma\in \Gamma}\Big(\big(T(x)\big)_\gamma\Big)^{c_\gamma}e^{-\big(T(x)\big)_\gamma},\\
-\log(LH(x))&=\sum_{\gamma\in \Gamma}  \big(T(x)\big)_\gamma - c_\gamma\log\Big(\big(T(x)\big)_\gamma \Big).
\end{align*}
One typically wants to find the maximum likelihood estimator $\hat x_{ML}$ that minimizes the negative log-likelihood,
$f(x):= \sum_{\gamma\in \Gamma}  \big(T(x)\big)_\gamma - c_\gamma\log\Big(\big(T(x)\big)_\gamma \Big),
$
i.e.~
\begin{equation*}
\hat x_{ML}=\arg\min_{ x}f( x) = \arg\min_{ x} \sum_{\gamma\in \Gamma}  \big(T(x)\big)_\gamma - c_\gamma\log\Big(\big(T(x)\big)_\gamma \Big).
\end{equation*}

To this end, we need the gradient, and preferably also Hessian in some manageable form, to be able to implement Gradient and Conjugate Gradient Descent schemes, or even Quasi-Newton ones, even if the detector is large (modern detectors can have $N\approx 10^4$ so we have to handle around $10^8$ variables). Below, we content with computing $df$ and $d^2f$. 
To begin, we consider 
\begin{equation*}
f(x+\Delta x)=\sum_{\gamma\in \Gamma}  \big(T(x)\big)_\gamma + \big(T(\dx)\big)_\gamma - c_\gamma\log\Big(\big(T(x)\big)_\gamma + \big(T(\dx)\big)_\gamma \Big).
\end{equation*}
Now, by Taylors formula we have 
\begin{equation*}
\log(a+\da) = \log(a) + \log(1+\da/a) = \log(a) +\frac{\da}{a} - \frac{1}{2} \left(\frac{\da}{a}\right)^2+\mathcal{O}(\da^3),\quad a>0,
\end{equation*}
where $\mathcal{O}$ denotes ``big ordo'', thus $f(x+\dx)$ can be expanded as 
\begin{align*}
&\sum_{\gamma\in \Gamma}  \big(T(x)\big)_\gamma  - c_\gamma\log\Big(\big(T(x)\big)_j \Big) + \big(T(\dx)\big)_\gamma - c_\gamma\left(\frac{\big(T(\dx)\big)_\gamma}{\big(T(x)\big)_\gamma} - \frac{1}{2} \frac{\big(T(\dx)\big)_\gamma^2}{\big(T(x)\big)_\gamma^2}\right)  +\\
& \mathcal{O}(||\dx||^3)=f(x) + \sum_{\gamma\in \Gamma}  \big(T(\dx)\big)_\gamma  - c_\gamma \frac{\big(T(\dx)\big)_\gamma}{\big(T(x)\big)_\gamma} + \frac{c_\gamma}{2} \frac{\big(T(\dx)\big)_\gamma^2}{\big(T(x)\big)_\gamma^2}  + \mathcal{O}(||\dx||^3).
\end{align*}
From this expression and the uniqueness part of Theorem \ref{t1}, it is clear that \begin{equation}\label{h4}
                                         df|_x(\Delta x)=\sum_{\gamma\in \Gamma}  \big(T(\dx)\big)_\gamma  - c_\gamma \frac{\big(T(\dx)\big)_\gamma}{\big(T(x)\big)_\gamma}
                                       \end{equation}
and \begin{equation}\label{h6}
     d^2f|_x(\Delta x,\Delta x)= \sum_{\gamma \in \Gamma}{c_\gamma} \frac{\big(T(\dx)\big)_\gamma^2}{\big(T(x)\big)_\gamma^2}.
    \end{equation}

At this point, we have identified simple expressions for $df|_x$ and $d^2f|_x$.
But it is still unclear how to obtain a gradient from \eqref{h4} and a ``matrix free'' Hessian from \eqref{h6}. In the next Section, we will see how this can be done.

\section{The gradient, the bilinear Hessian and the Hessian operator}\label{secghb}

We now provide a basis free definition of the gradient. Again, this idea is not new and can be found e.g.~in Chapter 2.6 in \cite{bauschke2017convex}, where it is developed in a general Hilbert space setting.

\begin{definition}\label{def4}
Let $\X$ be a space as in Assumption \ref{def1}. Given a functional $f:\X\rightarrow \R$, the gradient of $f$ at $x$ is the unique element $\nabla f|_x$ in $\X$ such that $df|_x(\dx)=\langle \nabla f|_x,\dx\rangle_{\R}$.
\end{definition}

The existence and uniqueness of $\nabla f|_x$ follows from the finite dimensional version of the Riesz Representation theorem. It is easy to see that \begin{equation}\label{hu8}
\nabla (f\circ \iota_\X)|_\bx=\iota_\X^{-1}(\nabla f|_x)
\end{equation}
where $x=\iota_\X(\bx)$ and $\circ$ denotes composition of functions, where the first $\nabla$ refers to the standard gradient for functions on $\R^n$, and the latter for the gradient in the sense of Definition \ref{def4}. Thus, the ``new'' definition is just the old one framed in such a way that we can avoid introducing a basis and work directly in $\X$.
We also remark that, in case $\X=\C^n$, $\nabla f|_x$ becomes a vector in $\C^n$ which is equal to the Wirtinger derivatives $\partial f/\partial \bold{z}$, which were popularized in optimization by the paper \cite{candes2015phase},
in which the new term ``Wirtinger flow'' was used for what is in fact simply gradient descent in $\C^n$.
Thus, the above definition can be viewed as an extension of this framework to multidimensional objects. Moreover, by relying on Taylor expansions rather than Wirtinger derivatives, arguably the derivation of the corresponding expressions becomes simpler.

\smallskip\indent{\sc Example (continued).\quad}
To illustrate, we continue the previous example and note that 
the linear term $df|_x(\dx)= \sum_{\gamma\in \Gamma}  \big(T(\dx)\big)_\gamma  - c_\gamma \frac{\big(T(\dx)\big)_\gamma}{\big(T(x)\big)_\gamma}  $ can be written
\begin{equation}
\left\langle  \mathbf{1} - \left(\frac{c_\gamma}{\big(T(x)\big)_\gamma}  \right)_{\gamma\in \Gamma}, T(\dx) \right\rangle_\R 
=
\left\langle T^*\left(  \mathbf{1} - \left(\frac{c_\gamma}{\big(T(x)\big)_\gamma}  \right)_{\gamma\in \Gamma}\right), \dx \right\rangle,
\end{equation}
where $\mathbf{1}$ denotes the identity element of $\R^{N\times N}$ and $\Gamma=\{1,\ldots, N\}^2$ as before, and the gradient is therefore
\begin{equation}
\nabla f|_x = T^*\left(  \mathbf{1} - \left(\frac{c_\gamma}{\big(T(x)\big)_\gamma}  \right)_{\gamma\in \Gamma}\right)
\end{equation}
which is an element of $\X$, where $T^*$ denotes the operator adjoint.\hfill$\diamond$
\smallskip

By definition, the operator adjoint is the linear operator such that 
\begin{equation}\label{u7}
\langle T(x),y\rangle =\langle x,T^*(y)\rangle
\end{equation}
holds for all $x$ and $y$. If $\X$ is $\R^n$ this is just the matrix transpose in our example, and if $\X=\C^n$ it is the Hermitian transpose. But in large scale applications it typically is the case that $T$ is implemented as a composition of operations acting on $\X$ directly, which is much faster to evaluate than via some matrix representation. In such cases it is crucial to derive an expression for $T^*$ of a similar form, so therefore we advice against thinking of $T^*$ as some sort of matrix transpose. 
Also, it is important to note that if $\X$ is a linear space over $\C$ then the above is true even if $T^*$ is the adjoint in $\X$ treated as a complex space, since by \eqref{realSP} we then have 
$$\langle T(x),y\rangle_\R =\Re\langle T(x),y\rangle =\Re \langle x,T^*(y)\rangle=\langle x,T^*(y)\rangle_\R.$$ 
Thus, in particular, if $T$ equals the discrete Fourier transform implemented via FFT, then $T^*$ becomes the IFFT. 


\subsection{The Bilinear Hessian and polarization}

We are now ready to formally introduce the Hessian as a bilinear form, which we did informally already in the introduction \eqref{hes}.
\begin{definition}\label{def5}
Given space $\X$ as in Assumption \ref{def1} and a functional $f:\X\rightarrow \R$, the symmetric bilinear form $d^2f|_x$ will be called ``the bilinear Hessian'' of $f$ at the point $x$, usually denoted $\H|_{x}$. 
\end{definition}
While the definition basically just changes the notation, it remains to be explained how to actually find it only knowing $d^2f|_x$ on the diagonal, which is what we get by following the procedure described in Section \ref{sec:taylor}.

\smallskip\indent{\sc Example (continued).\quad} We illustrate this with our example.
On the diagonal, we saw that the bilinear Hessian was given by \eqref{h6},
i.e. $$\H|_x(u,u)=  \sum_{\gamma\in\Gamma}\frac{c_\gamma}{\big(T(x)\big)_\gamma^2} \big(T(u)\big)_\gamma^2.$$
From this expression one can "guess" that the bilinear Hessian is
\begin{equation}\label{h7}
\H|_x(u,v)=  \sum_{\gamma\in \Gamma}\frac{c_\gamma}{2\big(T(x)\big)_\gamma^2} \big(T(u)\big)_\gamma \big(T(v)\big)_\gamma.
\end{equation}
Indeed, the above form is bilinear, symmetric and agrees with \eqref{h6} whenever $u=v$, and therefore by the uniqueness part of Theorem \ref{t1} we have found the correct expression. $\diamond$

\smallskip

In most applications, it is possible to copy the above idea, i.e.~to first guess and then refer to uniqueness. However, if this fails there is a mechanical way of deriving the expression, called ``polarization'', i.e.~the fact that by bilinearity we have that $\H|_x$ has to satisfy 
\begin{equation}\label{polarization}
\H|_x(u,v) = \frac{1}{2}\left( \H|_x(u+v,u+v) - \H|_x(u,u) - \H|_x(v,v)\right).
\end{equation}
It is easily verified that one can arrive at \eqref{h7} by plugging \eqref{h6} into the above formula.

\subsection{The Hessian as an operator}
\label{sec:HessOp}

It is a standard fact from linear algebra on $\R^n$ that all bilinear forms $\H$ on $\R^n$ can be written as $\H(\bold y,\bold z)=\langle \HH \bold y,\bold z\rangle$ for some matrix $\HH$, and that the form is symmetric if and only if the matrix is, i.e.~$\HH^t=\HH$ where $t$ denotes the matrix transpose.
In this section we generalize this observation to inner product spaces. Recall that real linear operators (which is the generalization of matrices to inner product spaces) were introduced in Definition \ref{def3}, and that operator adjoints were introduced in \eqref{u7} (which is the generalization of the matrix transpose). A real linear operator $H$ is called symmetric if $H^*=H$.
\begin{theorem}\label{t2}
Given a bilinear form $\H$ on some inner product space $\X$, there exists a unique linear operator $H:\X\rightarrow \X$ such that $\H(y,z)=\langle H(y),z\rangle$. The operator is symmetric if and only if the bilinear form is symmetric. Finally, setting $h(x)=\frac{1}{2}\H(x,x)$, we have $\nabla h|_{x}=H(x)$.
\end{theorem}
\begin{proof}
Introduce a basis in $\X$ and an isometric isomorphism $\iota_\X$ as in \eqref{iota}. By considering the bilinear form $\H(\iota_\X(\by),\iota_\X(\bold z))$ and the facts about such forms on $\R^n$ alluded to before the theorem, there exists a matrix $\bold A\in\R^{n\times n}$ such that $$\H(\iota_\X(\by),\iota_\X(\bold z))=\langle \bold A \by,\bold z\rangle_{\R^n}.$$
By inserting $\by=\iota_\X^{-1}(y)$ and similarly for $\bold z$ we see that $$\H(y,z)=\langle \bold A \iota_\X^{-1}(y), \iota_\X^{-1}(z)\rangle_{\R^n}=\langle \iota_\X\big(\bold A \iota_\X^{-1}(y)\big), z\rangle,$$ where we made use of the fact that $(\iota_\X^{-1})^*=\iota_\X$, which holds for all bijective isometries.
The sought operator $H$ is thus given by $H(y)=\iota_\X \big(\bold A \iota_\X^{-1}(x)\big)$, we leave the remaining details to the reader.
\end{proof}

If the symmetric bilinear form is a Hessian bilinear form $\H|_x$ coming from some functional $f$, we will denote the operator in the above theorem by $H|_x$ and call it the \emph{Hessian operator}. In analogy with \eqref{hu8} one easily sees that 
\begin{equation}\label{hu9}
(\nabla^2 (f\circ \iota_\X)|_\bx) \by=\iota_\X^{-1}(H|_x(y))
\end{equation}
holds for any $\bx=\iota_\X^{-1}(x)$ and $\by=\iota_\X^{-1}(y)$, where $\nabla^2 (f\circ \iota_\X)$ denotes the classical Hessian matrix of the function $f\circ \iota_\X$ on $\R^n$ and $(\nabla^2 (f\circ \iota_\X)|_\bx) \by$ denotes matrix vector multiplication. Thus, again, we are not introducing something new but just a framework in which one can avoid introducing bases and work directly in $\X$.

\smallskip\indent{\sc Example (continued).\quad}
We show how to find this operator in our example, for other explicit examples the reader is referred to \cite{carlsson2025efficient} or \cite{troedsson2024gradient,troedsson2024joint}. One of the reasons why such an expression can still be useful is that we can then always solve the equation $z=H(y)$ approximately by some iterative solver, which opens the door to quasi-Newton methods. To find it, we just take the bilinear Hessian and try to isolate one of the variables, say $z$, by moving everything to the right hand side. From \eqref{h7} we have
\begin{equation}
\H|_x(y,z)=  \sum_{\gamma\in \Gamma}\frac{c_\gamma}{2\big(T(x)\big)_\gamma^2} \big(T(y)\big)_\gamma \big(T(z)\big)_\gamma=\left\langle \frac{c}{2T(x)^2}\cdot T(y),T(z)\right\rangle
\end{equation}
where $\frac{c}{2T(x)^2}$ denotes the matrix $\left(\frac{c_\gamma }{2\big(T(x)\big)_\gamma^2}\right)_{\gamma\in \Gamma}$ and $\cdot$ stands for Hadamard multiplication. Moving $T$ to the other side, Theorem \ref{t2} shows that
\begin{equation}\label{Hessianop}
H|_x(y)=  T^*\left(\frac{c}{2T(x)^2}\cdot T(y)\right)
\end{equation}
is the abstract counterpart of the Hessian matrix.
Finally, we remark that the complexity for computing the bilinear Hessian and Hessian operator  \eqref{h7} and  \eqref{Hessianop} for $n=N^2$ datapoints is \emph{optimal} in the sense that it is identical to the complexity of the objective functional evaluation, i.e., essentially the forward operator used in this model, which if implemented with FFT equals $\mathcal{O}(n\log(n))$. 
The use of the Hessian in matrix form would, in contrast, require the computation and storage of $n^2$ elements, while its use (i.e., application to a vector), would also require $\mathcal{O}(n^2)$ operations.\hfill$\diamond$

\smallskip

\begin{algorithm}
\floatname{algorithm}{Procedure}
\caption{{Computation of Gradient, bilinear Hessian and Hessian operator of a function \( f(x) \) in \(\X\)}} \label{alg:MyAlgorithm}
\begin{algorithmic}

\State \textbf{INPUT:} Expression for function \( f(x),\;x  \in \mathbb{X} \)

\State \textbf{OUTPUT:} Expressions for
\begin{itemize}
\item[-]  \( \nabla f(x) \) (gradient)
\item[-]  \( \mathcal{H}|_x(\Delta y, \Delta z) \) (bilinear Hessian)
\item[-]  \( {H}|_x(\Delta y) \) (Hessian operator)\vspace{1mm}
\end{itemize}

\State \textbf{Step 1:} Perform a second-order Taylor expansion to approximate \( f(x + \Delta x) \).

\State \textbf{Step 2:} Separate the terms into:
\begin{itemize}
    \item the terms involving single \( \Delta x_\gamma \)'s: \( df|_x(\Delta x) \),
    \item the terms involving products \( \Delta x_{\gamma_1} \Delta x_{\gamma_2} \): \( \frac{1}{2} d^2 f|_x(\Delta x, \Delta x) \), so that 
\end{itemize}

\[
f(x + \Delta x) = f(x) + df|_x(\Delta x) + \frac{1}{2} d^2 f|_x(\Delta x, \Delta x) + O(|\Delta x|^3).
\]

\State \textbf{Step 3:} Identify an element \( v \in \mathbb{X} \) such that:
\[
df|_x(\Delta x) = \langle v, \Delta x \rangle_\mathbb{R};
\]
this element \( v \) is the \textit{gradient} \( \nabla f(x) \).

\State \textbf{Step 4:} Use polarization \eqref{polarization} to find the symmetric bilinear form \( \mathcal{H}|_x(\Delta y, \Delta z) \) such that:
\[
d^2 f|_x(\Delta x, \Delta x) = \mathcal{H}|_x(\Delta x, \Delta x);
\]
this form is the \textit{bilinear Hessian} \( \mathcal{H}|_x \).

\State \textbf{Step 5:} Find an element \( w(\Delta y) \in \mathbb{X} \) such that:
\[
\mathcal{H}|_x(\Delta y, \Delta z) = \langle w(\Delta y), \Delta z \rangle_\mathbb{R}
\]
This element \( w(\Delta y) \) is the \textit{Hessian operator} \( {H}|_x \) evaluated at $\Delta y$.

\end{algorithmic}
\end{algorithm}

\subsection{Approximate formulas and Automatic Differentiation}\label{secAD}

An alternative formula for computing the Hessian operator is to note that $H|_x(y)=\frac{d}{dt} \nabla f|_{x+ty},$ and along the same lines the bilinear Hessian can be expressed as 
\begin{equation}
\label{hess:ad}
\H(y,z)=\frac{d}{dt}\langle \nabla f|_{x+ty},z\rangle.
\end{equation}
To see this, note that by \eqref{hu8} and \eqref{hu9} we have, with $x=\iota_\X(\bx)$ and $y=\iota_\X(\by)$, that
\begin{align*}
&\iota_\X^{-1}\left(\frac{d}{dt} \nabla f|_{x+ty}\right)=
\frac{d}{dt} \left(\iota_\X^{-1}(\nabla f)|_{x+t y}\right)=\\&\frac{d}{dt} \left(\nabla (f\circ \iota_\X)|_{\bx+t \by}\right)= (\nabla^2 (f\circ \iota_\X)|_{\bx}) \by=\iota_\X^{-1}(H|_x(y))
\end{align*}
where $\nabla (f\circ \iota_\X)|_{\bx}$ and $\nabla^2 (f\circ \iota_\X)|_{\bx}$ denotes the classical gradient and Hessian matrix of the functional $f\circ \iota_\X$ on $\R^n$.

It is possible to compute both relying on Automatic Differentiation, a technique popular in machine learning for computing gradients without needing to derive explicit formulas. 
This was observed in \cite{pearlmutter1994fast} where it was used to provide a matrix free algorithm to compute the Hessian operator in the context of neural networks, and the idea has slowly spread to other fields, for example it was introduced to ptychography in \cite{kandel2021efficient}. 
Note that, upon evaluating $\frac{d}{dt}$ approximately with a finite difference, the above formulas provide fast alternatives to both AD and the analytic computations suggested in this paper, at the cost of numerical precision \cite{jerrell1997automatic}. This has also been tried in machine learning and it is noted in \cite{pearlmutter1994fast} that, while it sometimes works well, it is flawed by stability issues and also scales poorly when applied to large problems \cite{baydin2018automatic}. It is hard to objectively compare the pros and cons of manual computation, as suggested in this paper, versus AD. Clearly manual computation is more elegant and also easier to implement, once the math has been properly done, but at the end of the day computational speed is what matters, and here it gets hard to compare since it depends very much on which package is used for AD. According to the numerical tests in \cite{srajer2018benchmark} manual computation is the fastest and quite often by a large margin, see Table 2-4 along with a number of interesting runtime graphs. This is likely due to, at least in part, the computational overhead caused by AD dynamically constructing and traversing computational graphs \cite{baydin2018automatic,margossian2019review}. On the other hand, as noted in the introduction of \cite{baydin2018automatic}, ``manual computation is slow and prone to errors''. While this is certainly true, we have found that plotting graphs as in Figure \ref{fig:steprule} is an excellent tool for finding errors; any erroneous formula will immediately produce approximations that look off to the naked eye. Also, manual computing is slow once, AD is slower every time the algorithm is run. 

In the concrete case of the simple operator \eqref{Hessianop}, we implemented an AD version of it and found that using the formula \eqref{Hessianop} was roughly 3 to 6 times faster. For the tests, we considered large data volumes, with sizes ranging from $N=256$ to 8192 in each dimension. Formula \eqref{Hessianop} and its AD version were implemented using PyTorch, leveraging both CPU and GPU (CUDA) backends. In more complicated scenarios we have observed that the computational gain is even larger, which is also supported by the more extensive testing in the above references.

\section{Optimization methods}\label{sec_lg}

Let $f$ be a $C^2$-smooth function, and assume that we are interested in finding a local minimum $\hat x$, given an initial guess $x_0$ which, given a priori information, is not too far from $\hat x$. In this section we describe how to do Gradient Decent, Conjugate Gradient Descent and a Quasi-Newton method making use of the bilinear Hessian and the Hessian operator. 

Let $x_k\in \X$ denote an iterate in any of the above mentioned methods. The common denominator is that we will use the approximation 
\begin{equation}\label{Fapp}
  f(x_k+y)\approx f(x_k)+\Re \langle \nabla f|_{x_k},y\rangle +\frac{1}{2}\H|_{x_k}(y,y)
\end{equation} 
to compute the next iterate $x_{k+1}$. Note that \eqref{Fapp} is simply an abstraction of \eqref{secprx}.

\subsection{Gradient Descent and step size in first order methods}\label{secGD}

Given a certain descent direction $s_k$, we see that a near optimal strategy for step length for the step $x_{k+1}=x_k+\alpha_k s_k$ is to set $y=\alpha s_k $ in the right hand side of \eqref{Fapp} and minimize with respect to $\alpha$, which gives 
\begin{equation}\label{eqalpha}
               \alpha_k=-\frac{\langle \nabla f|_{x_k},s_k\rangle}{\H|_{x_k}( {s_k},s_k)}.
             \end{equation}
In particular, if $s_k=-\nabla f|_{x_k}$, we obtain the Gradient Descent method with Newton step-size;
\begin{equation}\label{GD}
  x_{k+1}=x_k-\frac{ \|\nabla f|_{x_k}\|^2}{\H|_{x_k}(\nabla f|_{x_k},\nabla f|_{x_k})}\nabla f|_{x_k}.
\end{equation}
Of course, the rule \eqref{eqalpha} is only applicable given that 
\begin{itemize}
\item[$i)$] The approximation \eqref{Fapp} is sufficiently accurate in a neighborhood of radius $|\alpha_k|\|s_k\|$ of $x_k$.
\item[$ii)$] The local geometry is convex, or at least that $\H|_{x_k}( {s_k},s_k)>0$.
\end{itemize}
If either of the above two assumptions fails, it is safer to perform a line-search, until one is close enough to a local minimum that the two conditions apply. In Section \ref{convGD} we prove that the above incarnation of Gradient Descent, henceforth denoted BH-GD where BH stands for "bilinear Hessian", does converge to a local minimum given that the initial point is close enough.

Fig.~\ref{fig:steprule} illustrates the quadratic approximation used in the step length computation, for one concrete realization of $T$ and $c$ in our main example from Section \ref{sec:poisson1}, and for iterations $1,3,6$ for the Gradient descent, as introduced below, applied to this realization. It can be observed that the approximation is quite accurate, and increasingly so as the iterations get closer to the minimum. We refer to Section \ref{secnum} for further details on the numerical experiment.

\begin{figure}
\centering
\includegraphics[width=0.33\linewidth]{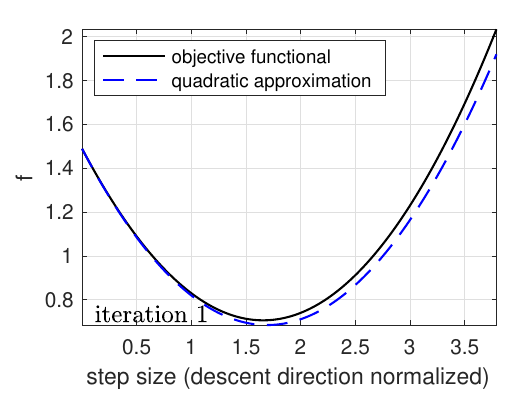}%
\includegraphics[width=0.33\linewidth]{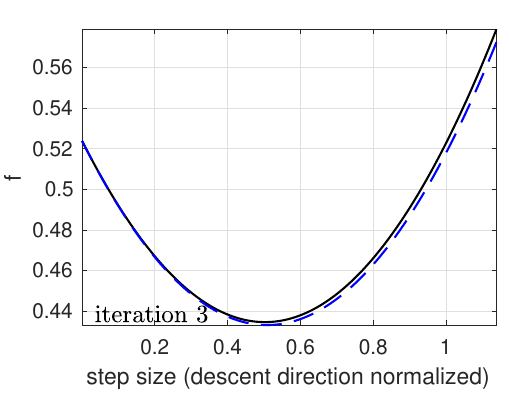}%
\includegraphics[width=0.33\linewidth]{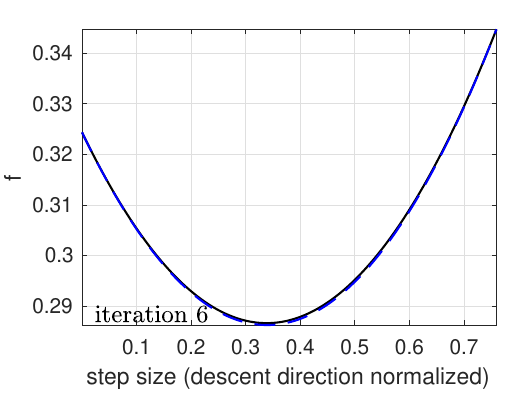}%
\caption{\label{fig:steprule}Illustration of step length computation using bilinear Hessian (BH-GD) for the example in Section \ref{sec:poisson1}.}
\end{figure}

\subsection{The Conjugate Gradient method}\label{seccg}

The original CG was developed by Hestenes and Stiefel \cite{hestenes1952methods} as a fast solver to equations of the form $\HH \bold x=\bold d$, where $\HH$ is a positive definite matrix on $\R^n$ and $\bold d$ represents some measured data, known to converge in only $n$ steps (ignoring numerical errors). Alternatively, in the real case at least, the method can be seen as a descent algorithm seeking to minimize the functional
\begin{equation}\label{e1}
f(\bx)=\frac{1}{2}\langle \bx, \HH \bx\rangle-\langle \bx,\bold d\rangle.
\end{equation}
The key idea is to pick search directions
\begin{equation}\label{e3}
\bold{s}_{k+1}=-\nabla f|_{\bx_{k+1}}+\beta_k \bold{s}_{k},
\end{equation}
where $\beta_k$ is chosen so that $\langle \bold{s}_{k+1}, \HH \bold{s}_{k}\rangle =0$, leading to the formula 
\begin{equation}\label{eqbetaD}
\beta_k=\frac{\langle \nabla f|_{\bx_{k+1}},\HH \bold{s}_{k}\rangle}{\langle \bold{s}_{k},\HH\bold{s}_{k}\rangle}.
\end{equation}
Returning to the case of a possibly non-quadratic $C^2$-functional $f$ on a general linear space $\X$ as in Assumption \ref{def1}, the formula for $\beta_k$ can be rewritten in a myriad of forms. The most famous ones are Fletcher-Reeves \cite{fletcher1964function}, Polak-Ribi\`{e}re \cite{polak1969note}, Hestenes-Stiefel \cite{hestenes1952methods}, dating back to the 50-60's, and more recently Dai-Yuan \cite{dai1999nonlinear} and Hager-Zhang \cite{hager2005new}. For example, the Fletcher-Reeves formula suggests to pick $\beta_k$ via $\frac{ \|\nabla f|_{\bx_{k+1}} \|^2}{ \|\nabla f|_{\bx_{k}} \|^2},$ which remarkably is identical with \eqref{eqbetaD} in the quadratic case. 
We refer to \cite{hager2006survey} for an overview and a more exhaustive list.

Each formula for $\beta_k$ is different and gives rise to a new incarnation of CG that usually bears the name of the inventors, and the jury is still out on which version is the best. The formula \eqref{eqbetaD} suggests the following general rule
for new search direction $s_{k+1}=-\nabla f|_{x_{k+1}}+\beta_k s_k$ by setting
\begin{equation}\label{eqbeta}
                          \beta_k=\frac{\H|_{x_{k+1}}(\nabla f|_{x_{k+1}},s_k)}{\H|_{x_{k+1}}( {s_k},s_k)}
\end{equation}
where $\H|_{x}$ as before denotes the bilinear Hessian of $f$ at $x$. This has been proposed by Daniel \cite{daniel1967conjugate} in 1967, and local convergence of the method in combination with the step-size rule \eqref{eqalpha} was claimed in \cite{daniel1967convergence},
 but the method has not gotten any traction since the Hessian is considered unnecessarily slow to evaluate, as discussed in the introduction. 
However, based on the developments in this paper, we see that Daniel's method can be implemented with a computational cost comparable with the other methods. The resulting method, in which the bilinear Hessian is used for both the $\alpha$ and $\beta$ parameters, will henceforth be denoted BH-CG. 

We remark that it is generally recommended to reinitialize the conjugate gradient direction at least when reaching the dimension of the underlying space, see e.g.~\cite{nesterov2013introductory}, whereas \cite{nazareth2009conjugate} goes one step further and claims that it is numerically advantageous to restart anytime the conjugate gradient direction starts to show ``unwanted'' behavior, such as not being a descent direction. We have seen this type of behavior when running BH-CG on intricate problems and/or when the initial guess is poor, so this is certainly something to consider when evaluating BH-CG. 

\subsection{Quasi-Newton methods}\label{secqn}

We finally consider how to use the Hessian operator for Newton and quasi-Newton methods. Upon differentiating the right hand side of \eqref{Fapp} we get $\nabla f|_{x_k}+H|_{x_k}(y)$ which equals zero whenever \begin{equation}\label{QNeq1}
                       y=-H|_{x_k}^{-1}(\nabla f|_{x_k}),
\end{equation}
given that $H|_{x_k}^{-1}$ exists. Of course, whether or not $H|_{x_k}^{-1}$ exists is, for the applications we have in mind, untractable to compute. However, assuming it is, we can still solve this equation using CG, as outlined in (\ref{e1} - \ref{eqbetaD}), on the functional
\begin{equation}\label{g5} g(y)=\frac{1}{2}\langle H|_{x_k}(y),y\rangle+\langle \nabla f|_{x_k},y\rangle=\frac{1}{2}\H|_{x_k}(y,y)+\langle \nabla f|_{x_k},y\rangle.\end{equation}
Note that the Hessian operator is needed in order to compute the gradient of $g$, via Theorem \ref{t2}.
Since the bilinear Hessian is a quadratic form, given any concrete basis, we know that this will terminate in a finite number of steps. The resulting method will be called BH-N. More precisely, it will terminate in the same amount of steps as there are dimensions in the underlying space $\X$, which for most applications is way to high a number. If we instead terminate the CG method after a much smaller number of steps, we get a new search direction $s_{k+1}$ which approximately solves \eqref{QNeq1}, where of course one can play with a multitude of stopping criteria for the inner loop. We refer to any such method as a Quasi-Newton method. The new point $x_{k+1}$ in the outer loop is then computed via $x_{k+1}=x_k+\alpha_k s_k$ as in Section \ref{secGD}. We call the resulting algorithm BH-QN. It is known to converge locally under favorable assumptions, see Section \ref{secloc}, but it is also known to be unstable, wherefore more stable variants like Gauss-Newton and Levenberg-Marquardt are popular as well. Of course these can also be easily implemented using the Hessian operator, but we do not take such into account in the present exposition.

\section{Numerical illustration}\label{secnum}

\begin{figure}
\centering
\includegraphics[width=0.49\linewidth]{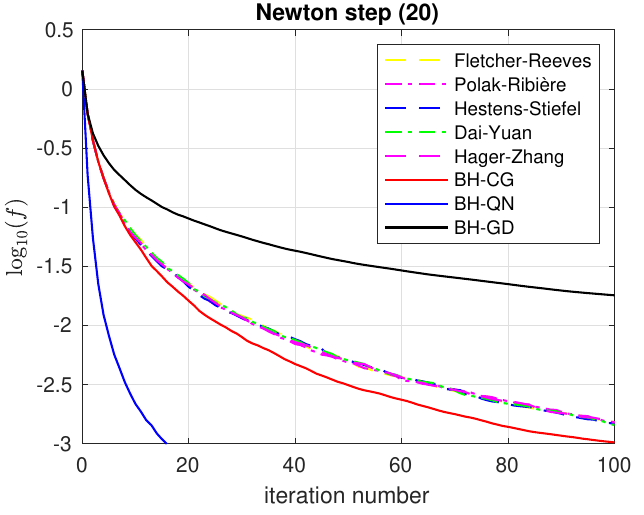}%
\includegraphics[width=0.49\linewidth]{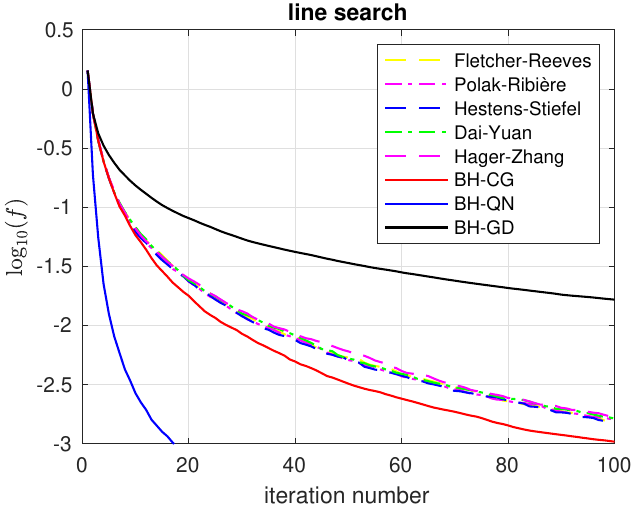}\\
\includegraphics[width=0.49\linewidth]{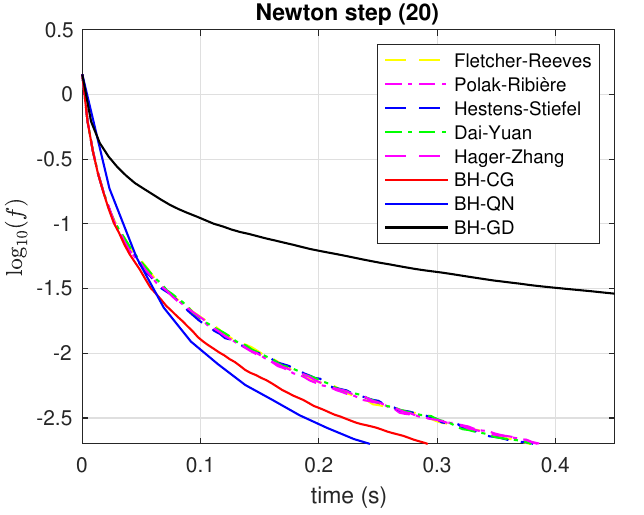}%
\includegraphics[width=0.49\linewidth]{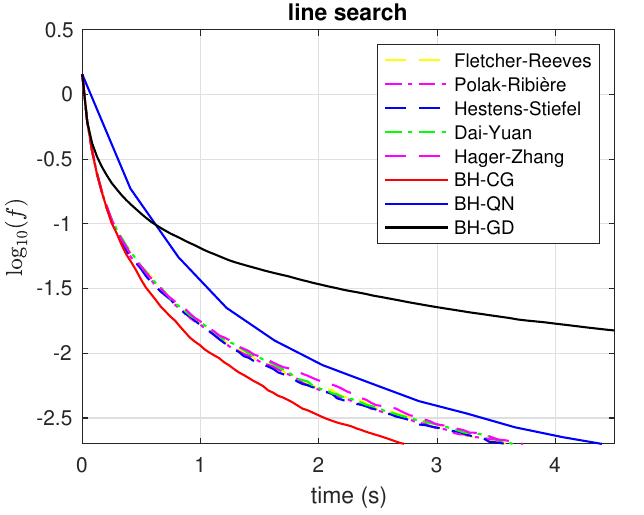}%
\caption{\label{fig:performance}Objective functional value vs. iteration number (top row) and vs. computation time (bottom row), respectively, for step length determined by quadratic approximation using the bilinear Hessian (left column) and by line search (right column, 50 objective function evaluations), respectively.}
\end{figure}

In this section we consider a concrete instance of the main example from Section \ref{sec:poisson1}. We set $N=100$, draw independent uniformly distributed rate parameters $(\bx)_i\sim \mathcal{U}([1,5])$, and let $T$ be a convolution operator (Gaussian blur) implemented via FFT. Performance results reported below are averages (median values) obtained for $100$ independent realizations for $\bx$.

Fig.~\ref{fig:steprule} illustrates the quadratic approximation used in the step length computation, for one single realization, and for iterations $1,3,6$ of the Gradient method as introduced in Section \ref{secGD}. 
Fig.~\ref{fig:performance} reports median objective functional values (in $\log_{10}$ scale) for five different standard CG rules (Fletcher-Reeves, Polak-Ribi\`ere, Hestens-Stiefel, Dai-Yuan, Hager-Zhang), along with Daniel's method using the bilinear Hessian (denoted BH-CG), gradient descent (BH-GD), and quasi-Newton using $12$ inner iterations (BH-QN). All algorithms make use of the same step length rule: proposed (left column), line search (center column) and an overlay thereof (right column); line search is performed for $50$ points logarithmically spaced in the interval $(0.1 \alpha, 3.3\alpha)$, where $\alpha$ is the value given by our step length rule. Results are plotted as a function of iteration number (top row) and computation time (bottom row), respectively, and lead to the following conclusions.
First, gradient descent has slowest convergence (both in terms of iterations and time) and will not be further discussed. 
Second, quasi-Newton yields best performance for this problem and converges in roughly $\sim 1/8$ of the iterations and $\sim 2/3$ of the time needed by the classical CG rules.
Third, the classical CG rules yield similar performance, but are outperformed by a marging of about 30\% (in iteration number and time) by the proposed BH-CG method. Moreover note that restart is needed for the classical CG methods ($\beta$ is set to zero whenever $\bold{s}_{k+1}$ would not be a descent direction) to yield this performance level; without restart, some of those methods are even found to diverge for this problem. In contrast, no restart is used or needed for BH-CG, indicating increased stability and robustness of our proposed method for computing $\beta$.

Finally, the proposed step length rule is highly effective: in terms of iteration numbers, it yields an almost identical performance as line search; in terms of computation time, it is about $10$ times faster than line search (with $50$ objective functional evaluations).

\section{Local convergence results}\label{secloc}

It is easy to construct examples where the algorithms presented in Section \ref{sec_lg} will diverge if initialized too far from a local minimum. The only thing that puts our Gradient Descent apart from the standard incarnation is the step length, and it is conceivable that a proof can be constructed along the lines described in Section 1.2.3 of \cite{nesterov2013introductory}. For completeness, we do the details below. Section 1.2.4 of the same book then proves local quadratic convergence for the Newton method, which can be applied directly in our setting since the BH-N does not rely on any step length, and in Section 1.3.2 it is claimed that Conjugate Gradient has ``local $n$-step quadratic convergence'', which means that $\|x_n-\hat x\|^2\leq const\cdot\|x_0-\hat x\|$ given that the initial point $x_0$ is close enough to a local minimum $\hat x$.

However, there is no reference for this claim, and books and overview articles such as \cite{andrei2020nonlinear} and \cite{hager2006survey} do not mention such results. In fact, both these works are mainly concerned with convergence proofs assuming that the step length satisfies certain Wolfe conditions, which leads to the question of whether our formula \eqref{eqalpha} will satisfy this, asymptotically. It turns out that J.W. Daniel has published results on convergence (albeit not quadratic convergence) of various CG incarnations in \cite{daniel1967conjugate}, and in fact that he also considered the precise version BH-CG in \cite{daniel1967convergence} from 1967. However, in \cite{daniel1970correction} (from 1970) a correction to \cite{daniel1967conjugate} appears. Despite this, in 1972, A. L. Cohen \cite{cohen1972rate} points out that the arguments in Daniel's three contributions contain several errors, and that the same goes for prior proof-attempts by Polyak. However, he also points out that \textit{if} Daniel would have considered Conjugate Gradient with ``restart'', which means that after usually $n$ iterations one reinitializes the previous direction to zero, then ``several of his erroneous results would have been correct''. But whether this applies to the results in \cite{daniel1967convergence} remains nonetheless unclear. Also, the main theorem in \cite{cohen1972rate} does not apply to BH-CG since it assumes exact line search, so we have not found a proof that this precise version converges, and we leave this question for future research. That said, based on our numerical tests we are lead to believe that it always converges super-linearly when initialized sufficiently close to a local minimum.

Concerning Newton's method BH-N, as presented in \ref{secqn}, this does not rely on any step length and therefore coincides exactly with the classical incarnation. There is of course an abundance of classical results giving conditions on local convergence of this method, see e.g.~\cite{ball1974note} where it is shown that the method always converges to $\hat x$ for all starting points $x_0$ in a ball around $\hat x$, given the same assumptions on $f$ as in Theorem \ref{t3} below. The same conclusion, but with quadratic convergence rate, is shown in Theorem 1.2.5 of \cite{nesterov2013introductory}, and even faster convergence is demonstrated, under additional assumptions, in \cite{gragg1974optimal}. We therefore refrain from adding any new proof for BH-N.

\subsection{Gradient Descent}\label{convGD}

In this section we prove local convergence of BH-GD, i.e.~the Gradient Descent with Newton step length rule from Section \ref{secGD}. There is an abundance of similar theorems in the literature, given various step length constraints, see for example Theorem 1.2.4 of \cite{nesterov2013introductory}. However, we have not found a theorem which explicitly states local convergence when using the Newton step-size rule, so we include one for completeness. As usual, we let $\X$ be a linear space as in Assumption \ref{def1} and we let $f:\X\rightarrow \R$ be an objective functional. Since $\X$ is finite dimensional, we have that $f$ is strongly convex at a point $x$ if and only if $\H|_x(y,y)>0$ for all $y\neq 0$.

\begin{theorem}\label{t3}
Let $f\in C^2(\X)$ be given, let $\hat x$ be a local minima such that $f$ is strongly convex at $\hat x$. Then there exists a ball $B$ around $\hat x$ such that BH-GD converges (with linear rate) to $\hat x$ for all initial points $x_0\in B$. 
\end{theorem}

The proof of Theorem \ref{t3} will rely on a number of lemmas. First of all note that by considering $x\mapsto f(x-\hat x)-f(\hat x)$, we see that there is no restriction to assume that $\hat x=0$ and that $f(0)=0$. Secondly, upon introducing a basis as in \eqref{iota} and identifying $\X$ with $\R^n$, we may as well work with $\iota_\X^{-1} f\iota_\X$ (see the proof of Theorem \ref{t1} for details). Thus, we shall assume that $\X=\R^n$ to begin with.
Thirdly, we introduce the notation $\hat{H}:= H|_0$ for the Hessian of $f$ at 0, since this will play a special role in the proof.
Since we have assumed that $f$ is strongly convex at the minimum, it follows that $$\|x\|_{\hat{H}}:= \sqrt{\hat{\H}(x,x)}$$ defines a norm on $\R^M$ which we will refer to as the $\hat{H}$-norm. 
The general idea of the proof is to show that \begin{equation}\label{theta}
                                                \|x_{k+1}\|_{\hat{H}}\leq \theta \|x_k\|_{\hat{H}}
\end{equation} for some $\theta<1$, given that $\|x_k\|_{\hat{H}}$ is less than some given number $R_0$. Since all norms on finite-dimensional spaces are comparable, this immediately entails that $\lim_{k\rightarrow\infty} x_k=0$.

By continuity of the second derivatives (and again the fact that all norms on finite dimensional spaces are comparable), we can pick $\delta>0$ such that the operator norm of $H|_{x}-\hat{H}$ is less than $\epsilon$ whenever $\|x\|_{\hat{H}}<\delta$, \textit{where the operator norm can refer either to the one induced by the Euclidean norm or the $\hat{H}$-norm}. We assume in the remainder that $\epsilon$ and $\delta$ are such that this holds, where the concrete value of $\epsilon$ will be fixed later.

\begin{proposition}\label{pp1}
We have $\Big|f(x)-\|x\|_{\hat{H}}^2\Big|\leq \epsilon \min(\|x\|^2,\|x\|_{\hat{H}}^2)$ and $$\Big|\H|_x(y,y)-\|y\|_{\hat{H}}^2\Big|\leq \epsilon\min(\|y\|^2,\|y\|_{\hat{H}}^2)$$ whenever $\|x\|_{\hat{H}}<\delta$. Moreover $\|\nabla f|_{x}-\hat{H} x\|_{\hat{H}}\leq \epsilon \|x\|_{\hat{H}}.$
\end{proposition}

\begin{proof}
Since $\|y\|_{\hat H}^2=\langle \hat{H} y,y\rangle$ we get $$\sup_{y\neq 0}\frac{|\H|_{x} (y,y) -\|y\|^2_{\hat{H}}|}{\|y\|^2}=\sup_{y\neq 0}\frac{|\langle (H|_{x}-\hat{H}) y,y\rangle|}{\|y\|^2}\leq\sup_{y\neq 0}\frac{ \|H|_{x}-\hat{H}\| \|y\|^2}{\|y\|^2}<\epsilon.$$ The same computation holds with $\|y\|_{\hat{H}}$ in place of $\|y\|$ on the bottom row, which proves the second inequality. By Taylors formula (at 0) we have $f(x)=\langle H_{\xi} x,x\rangle$ for some $\xi$ on the line connecting 0 and $x$, which gives that 
$$f(x)-\|x\|_{\hat H}^2=\langle (H|_{\xi}-\hat{H}) x,x\rangle$$ and thus the first follows in the same way.
Finally, we have $\nabla f(x)=\int_{0}^1 \frac{d}{dt}\nabla f(tx) dt= \int_{0}^1 H|_{tx}(x) dt$. Swapping $H|_{tx}$ for $\hat{H}$ turns this into $\hat{H}x$ so
$$\|\nabla f|_x-H|_0x\|\leq \int_{0}^1 \|(H|_{tx}-\hat{H})(x)\| dt\leq \int_{0}^1 \epsilon \|x\| dt=\epsilon \|x\|,$$ as desired.
\end{proof}

We assume from now on that $\epsilon>0$ has been fixed and that $x$ is such that $\|x\|_{\hat{H}}<\delta$. Note that the Newton step length is given by $$\alpha_x=\frac{\|s_x\|^2}{\H|_x(s_x,s_x)}$$ with $s_x=-\nabla f|_{x}$. Let $M$ and $m$ be such that \begin{equation}\label{eqmM}m\|y\|\leq \|H|_x y\|\leq M\|y\|\end{equation}
holds for all $x$ with $\|x\|_{\hat{H}}<\delta$. By the continuity of eigenvalues of $H|_x$ and the assumption that $f$ is strongly convex at 0 it is clear that we can assume that $m>0$, given that $\delta$ is small enough. Note that 
$$\frac{1}{M}\leq \alpha_x\leq \frac{1}{m}$$
by Proposition \ref{pp1}, and also that \begin{equation}\label{y6}
                                          f(x+\alpha_xs_x)\approx f(x)-\alpha_x \|s_x\|^2+\frac{\alpha_x^2}{2}\H|_x(s_x,s_x)=f(x)-\frac{\alpha_x \|s_x\|^2}{2}.
                                        \end{equation}
Since $\alpha_x$ stays away from 0 we can use this to show that the function-value at $x+\alpha_xs_x$ is always lower than at $x$.

Now for the details, note that for any $\alpha\in [0,\alpha_x]$ we have $$\|x+\alpha s_x\|_{\hat{H}}\leq \|x\|_{\hat{H}}+\alpha \|s_x\|_{\hat{H}}\leq \|x\|_{\hat{H}}+\frac{1}{m}(\|\hat{H}x\|_{\hat{H}}+\epsilon\|x\|_{\hat{H}})\leq (1+\frac{M+\epsilon}{m})\|x\|_{\hat{H}},$$ where we used that $\|\hat{H}x\|_{\hat{H}}^2=\langle \hat{H}^3x,x\rangle\leq M^2\langle \hat{H}x,x\rangle=M^2\|x\|_{\hat{H}}^2$.
We pick $R_0=\delta/(1+\frac{M^2+\epsilon}{m})$ and assume that $\|x\|_{\hat{H}}\leq R_0,$ which thus ensures that Proposition \ref{pp1} applies to $x+\alpha s_x$ for any $\alpha$ as above.
By Proposition \ref{pp1} we then get \begin{equation}\label{y8}
                                   f(x+\alpha_x s_x)\geq \|x+\alpha_x s_x\|_{\hat{H}}^2-\epsilon\|x+\alpha_x s_x\|_{\hat{H}}^2.
                                 \end{equation}
and moreover by Taylor's formula we have
$$f(x+\alpha_xs_x)= f(x)-\alpha_x \|s_x\|^2+\frac{\alpha_x^2}{2}\H|_{x+\alpha s_x}(s_x,s_x)$$ for some $\alpha\in[0,\alpha_x]$ which combined with \eqref{y6} gives that
\begin{align*}
 & f(x+\alpha_xs_x)= f(x)-\frac{\alpha_x \|s_x\|^2}{2}+\frac{\alpha_x^2}{2}|(\H|_{x+\alpha s_x}-\H|_x)(s_x,s_x)|\leq \\
 & \|x\|_{{\hat{H}}}^2+\epsilon\|x\|_{{\hat{H}}}^2 -\frac{\|s_x\|^2}{2M}+\frac{2\epsilon}{2m^2}\|s_x\|^2\leq
 \|x\|_{{\hat{H}}}^2+\epsilon\|x\|_{{\hat{H}}}^2 -\left(\frac{1}{2M}-\frac{\epsilon}{m^2}\right)\|s_x\|^2
 \end{align*}
As in the previous section we may assume that $\epsilon$ is small enough that the final parenthesis is positive, which since \begin{equation}\label{ki8}
\|s_x\|=\|\nabla f|_x\|\geq \frac{\|\nabla f|_x\|_{\hat{H}}}{\sqrt{M}}\geq \frac{\| \hat{H} x\|_{\hat{H}}-\epsilon\|x\|_{\hat{H}}}{\sqrt{M}}\geq \left(\frac{m}{\sqrt{M}}-\frac{\epsilon}{\sqrt{M}}\right)\|x\|_{\hat{H}}
\end{equation} implies that
\begin{align*}
&f(x+\alpha_xs_x)\leq  \left( 1+\epsilon -\left(\frac{1}{2M}-\frac{\epsilon}{m^2}\right)\left(\frac{m}{\sqrt{M}}-\frac{\epsilon}{\sqrt{M}}\right)^2\right)\|x\|_{\hat{H}}^2 
\end{align*} 
Combined with \eqref{y8} this gives $$\|x+\alpha_x s_x\|_{\hat{H}}^2\leq (1-\epsilon)^{-1}\left( 1+{\epsilon} -\left(\frac{1}{2M}-\frac{\epsilon}{m^2}\right)\left(\frac{m}{\sqrt{M}}-\frac{\epsilon}{\sqrt{M}}\right)^2\right)\|x\|_{\hat{H}}^2.$$
Since for $\epsilon=0$ the constant equals $1-\frac{m^2}{2 M^2}$ we can pick $\epsilon$ such that this is less than some $\theta<1$, as desired.

\section{Conclusion}\label{sec13}

In this work, we introduced a novel framework for computing second-order information in smooth optimization problems, with a particular emphasis on efficiency for large-scale applications. The framework is computationally efficient and easy to implement, offering direct formulas for the gradient and Hessian that can be broadly applied across various optimization tasks. We have provided a detailed explanation of each step and the corresponding proofs, ensuring that the framework is both complete and accessible to researchers from diverse fields. While some of the formula derivations are not new, the method as a whole offers a novel coherent and practical approach. To further facilitate its adoption, we presented a clear, step-by-step procedure for applying the method, which involves Taylor expansions, term reordering, and polarization.

In our numerical demonstration, the method for evaluating the Hessian in a 2D convolution problem with Poisson noise was 3 to 6 times faster than the one based on automatic differentiation from PyTorch. This speedup was consistent across data sizes ranging from 256×256 to 8192×8192, on both modern CPU and GPU. We also observed similar acceleration when solving other, more complex optimization problems. Furthermore, the method achieved faster convergence and more stable behavior in the Conjugate Gradient method when computing the conjugate direction using our Hessian formulas, compared to classical methods like those from Dai-Yuan and Polak-Ribiere, with the step size determined using our approach. Additionally, we showed that our method enabled efficient evaluation of the Hessian operator for Quasi-Newton methods, resulting in faster convergence. Although Quasi-Newton methods generally offer faster convergence, they tend to be less stable than Conjugate Gradient methods. Therefore, the choice of method should be based on the specific needs of the application, with testing recommended to identify the most suitable approach.


Our numerical experiments highlighted the potential of this method for large-scale applications, such as X-ray ptychography~\cite{kandel2021efficient} and holotomography~\cite{nikitin2024x,nikitin2025single}, where the problems often involve on the order of $10^8$ variables. We plan to study this method for such problems in future work. Notably, the proposed framework offers a clear path to solving these highly challenging, contemporary problems with reduced computational cost and faster convergence, while convergence guarantees further enhance its practicality.

Looking further ahead, this framework holds significant promise for a wide range of applications beyond traditional optimization tasks. Its adaptability to machine learning makes it a strong candidate for accelerating optimization in deep learning and other ML models. Moreover, the method’s potential for automation, through the automated application of Taylor expansions and polarization, opens up exciting possibilities for fully automated optimization pipelines. We envision it as a viable alternative to widely used automatic differentiation methods in ML applications. Additionally, we aim to extend the framework to handle more complex objective functions, where the bilinear form approximation could be generalized or improved. Finally, a deeper exploration of the automation process and its integration with existing ML frameworks, such as PyTorch and Tensorflow, will be essential for scaling this approach to even broader applications.

\hfill \break
\noindent\textbf{Acknowledgments} Work in part supported by Swedish Research Council Grant 2022-04917.
This research used resources of the Advanced Photon Source, a U.S. Department of Energy (DOE) Office of Science user facility at Argonne National Laboratory and is based on research supported by the U.S. DOE Office of Science-Basic Energy Sciences, under Contract No. DE-AC02-06CH11357.

\hfill \break
\noindent\textbf{Data Availability} The data that support the finding of this study are available from the corresponding author upon reasonable request.

\section*{Declarations}
\textbf{Conflict of interest} The authors have no relevant financial or non-financial interests to disclose.

\bibliography{sn-bibliography}

\end{document}